\documentclass{article}
\usepackage{geometry}               
\usepackage[english]{babel}
\usepackage[utf8]{inputenc}
\usepackage{amssymb,amsmath,amsthm,amsfonts}
\usepackage{graphicx}
\usepackage[usenames,dvipsnames]{xcolor}
\usepackage{cancel}
\usepackage[colorlinks]{hyperref}
\usepackage{enumitem}
\usepackage{comment}

\usepackage[sc]{mathpazo}

\newcommand{\N}{\mathbb{N}}

\newcommand{\R}{\mathbb{R}}

\newcommand{\eps}{\varepsilon}

\def\XXint#1#2#3{{\setbox0=\hbox{$#1{#2#3}{\int}$ }
\vcenter{\hbox{$#2#3$ }}\kern-.6\wd0}}

\newtheorem{proposition}{Proposition}[section]
\newtheorem{theorem}[proposition]{Theorem}

\newtheorem{lemma}[proposition]{Lemma}

\theoremstyle{definition}
\newtheorem{definition}[proposition]{Definition}
\newtheorem{remark}[proposition]{Remark}

\numberwithin{equation}{section}
\newcommand{\beq}{\begin{equation}}
\newcommand{\eeq}{\end{equation}}
\newcommand{\ben}{\begin{enumerate}}
\newcommand{\een}{\end{enumerate}}
\newcommand{\bit}{\begin{itemize}}
\newcommand{\eit}{\end{itemize}}

\title{Normalized solutions for the Sobolev critical\\ Schr\"{o}dinger equation with trapping potential}

\author{Junwei Yu}

\date{\today}

\begin{document}
\maketitle

\begin{abstract}
We study the existence and multiplicity of positive normalized solutions with prescribed $L^{2}$-norm
for the Sobolev critical Schr\"odinger equation
\begin{equation*}
\begin{cases}
-\Delta U +V(x)U = \lambda U + |U|^{2^{*}-2}U  & \quad  \text{  in  } \R^N,\smallskip\\
\displaystyle \int_{\mathbb{R}^{N}} U^2\,dx = \rho^{2},
\end{cases}
\end{equation*}
where $N \ge 3$, $V\ge 0$ is a trapping potential, $\lambda \in \R$ and $2^*=\frac{2N}{N-2}$.
Our first result is that the existence of local minimum solutions for $\rho \in (0, \rho^*)$, for some suitable $\rho^* > 0$, under appropriate assumptions on the potential. These solutions correspond to ground states. Our second result concerns the existence of mountain pass solutions, under the same assumptions.

\end{abstract}
\noindent
{\footnotesize \textbf{AMS-Subject Classification}}.
{\footnotesize 35J20, 35B33, 35Q55, 35Q89, 35J61.}

\section{Introduction}

In this paper, we study the existence and multiplicity of standing waves of prescribed $L^{2}$-norm for the evolutive Sobolev critical Schr\"odinger equation
\begin{equation}\label{eq:evolutive}
i \Phi_{t}+\Delta \Phi-V(x)\Phi  + |\Phi|^{2^{*}-2}\Phi=0, \quad (t,x)\in \R\times \R^N,
\end{equation}
where $\Phi = \Phi(t,x): \R \times \R^N \ \to \mathbb{C}$, denotes the wave function, $V$ is a (real valued) potential, $N \ge 3$ and $2^{*}=\frac{2N}{N-2}$ is the Sobolev critical
exponent.

We recall that standing waves to \eqref{eq:evolutive} are solutions of the form
\[
\Phi(t,x) = e^{-i \lambda t} U(x), \quad (t,x)\in \R\times \R^N,
\]
where $U$ is a real function and $\lambda \in \R$. This ansatz leads to the elliptic problem
\[
-\Delta U +V(x)U = \lambda U + |U|^{2^{*}-2}U \qquad U  \text{  in  } H^1(\R^N), \quad \lambda \in \R.
\]
Such equations appear in various physical phenomena, including nonlinear optics and the theory of Bose-Einstein condensation. 
To study the existence and qualitative properties of solutions, one of the most powerful tools is the variational method, which allows us to interpret the solutions as critical points of suitable functionals associated with the equation.

There are mainly two variational approaches that have been developed in this framework. 
The first one considers the parameter $\lambda$ fixed and aims to find the critical points of the corresponding action functional $\mathcal{A}_\lambda$ : $H^{1}(\R^N) \rightarrow \mathbb{R}$, defined by
\[
\mathcal{A}_\lambda(U) 
= \frac12 \int_{\R^N} \big(|\nabla U|^2 + (V-\lambda)|U|^2 \big) - \frac1p \int_{\R^N} |U|^p.
\]
as in \cite{BerestyckiLions1983}. The second approach, following \cite{CazenaveLions1982}, treats $\lambda$ as an unknown parameter and focus instead on the energy functional $\mathcal{E}$ : $\mathcal{Q}_{\rho} \rightarrow \mathbb{R}$, defined by
\[
\mathcal{E}(U)=\frac{1}{2}\int_{\mathbb{R}^{N}}|\nabla U|^{2}dx + \frac{1}{2}\int_{\mathbb{R}^{N}}V(x) U^{2}dx-\frac{1}{2^{*}}\int_{\mathbb{R}^{N}}|U|^{2^{*}}dx
\]
constrained to the $L^{2}$-sphere
\begin{equation*}
\mathcal{Q}_{\rho}=\left\{U\in H^{1}(\mathbb{R}^{N}): \int_{\mathbb{R}^{N}} U^{2}=\rho^{2}\right\},
\end{equation*}
where $\lambda$ appears as a Lagrange multiplier. Under the assumptions imposed on the potential $V(x)$ and the mass $\rho$, these two approaches lead to distinct existence and multiplicity results, each requiring delicate analytical techniques. 
In this paper, we adopt the second method, as provides information on the orbital stability of the associated standing waves: local minimizers usually give rise to orbitally stable sets, whereas saddle points should correspond to unstable solitons.

Let $V=0$, we consider the free pure power energy $\mathcal{E}_{0}$ : $\mathcal{Q}_{\rho} \rightarrow \mathbb{R}$, defined by
\[
\mathcal{E}_{0}(U)=\frac{1}{2}\int_{\mathbb{R}^{N}}|\nabla U|^{2}dx -\frac{1}{2^{*}}\int_{\mathbb{R}^{N}}|U|^{2^{*}}dx,
\]
constrained to $\mathcal{Q}_{\rho}$. Define the scaling 
\[
\R^+\ni h \mapsto U_h(x) := h^{N/2}U(hx) \in \mathcal{Q}_\rho,
\]
one obtains
\[
\mathcal{E}_{0}(U_{h}) = \frac{h^{2}}{2}\int_{\R^{N}} |\nabla U|^2\, dx-\frac{h^{2^{*}}}{p}\int_{\R^{N}} |U|^{2^{*}}\, dx,
\]
which yields $\mathcal{E}_{0}$ has a mountain pass geometry on $\mathcal{Q}_\rho$. In general, the mountain pass geometry may give rise to two types of critical points. More precisely, such geometry provides Palais-Smale sequences at two distinct energy levels: a mountain pass level and a local minimum level. Actually, local minimizing sequences are not necessarily Palais-Smale sequences, but one can apply Ekeland’s variational principle to construct a new Palais-Smale sequence at the local minimum level.

For $\mathcal{E}_{0}$ on $\mathcal{Q}_{\rho}$, the sequence associated to the local minimum level always converges weakly in $H^{1}(\R^{N})$ to $0$, and hence does not yield a normalized solution. On the other hand, the mountain pass solution exists only in dimension $N \ge 5$, since the scaled Aubin-Talenti functions are not in $L^{2}(\R^N)$ for $N=3,4$.

It is natural to expect that the mountain pass geometry can be extended to perturbations of $\mathcal{E}_{0}$, such as those involving combined nonlinearities, different domains, non-constant potentials or more general nonlinearities. Consequently, under these perturbations, two natural questions arise: does a local minimizer exist? And does a mountain pass solution exist, even in dimensions $N=3,4$?

Indeed, in some cases in the literature it has been proved that the answers to the above questions are positive. In the Sobolev critical case, the perturbation which is given by combined nonlinearities, was first introduced by Soave \cite{MR4096725}. Soave proved the existence of a local minimizer and left the existence of a mountain-pass solution as an open question. Later, the problem was solved by Jeanjean and Le \cite{MR4476243} for $N \geq 4$ and by Wei and Wu \cite{MR4433054} for $N=3$. For related results in the Sobolev critical case, we refer the reader to studies on bounded domains \cite{chang2025positivenormalizedsolutionsschrodinger,MR4847285,song2024positivenormalizedsolutionsstarshaped}, on systems \cite{MR3918087}, and on potentials \cite{verzini2025normalizedsolutionsnonlinearschrodinger}.

Notice that the Sobolev subcritical case has been investigated much more. Since the seminal work of Jeanjean \cite{MR1430506}, the analysis of normalized solutions to Schr\"odinger equations has attracted considerable attention in the last ten years. While below we provide a more detailed discussion of the literature concerning different types of perturbations, from this general perspective we mention here only a few recent works, and refer the interested reader to the references therein for further contributions: e.g. equations on bounded domains(step well
trapping potential) \cite{ntvAnPDE,MR3689156}; about combined nonlinearities \cite{MR4107073}; equations with potentials \cite{MR4304693,MR4443784}; equations on metric graphs \cite{MR4404069,MR4132757,MR4601303,MR4371084,MR4241295}; equations on product space \cite{MR3219500,pierotti2025energylocalminimizersnonlinear}.

In the present paper, we deal with the perturbation: $V$ is a trapping potential. Therefore, we consider the problem
\begin{equation}\label{eq:main}
\begin{cases}
-\Delta U +V(x)U = \lambda U + |U|^{2^{*}-2}U  & \text{in }\mathbb{R}^{N},\smallskip\\
\displaystyle \int_{\mathbb{R}^{N}} U^2\,dx = \rho^{2},
\end{cases}
\end{equation}
in the unknown $(U,\lambda)\in H^1(\R^N)\times\R$, where $\rho>0$, $N\geq3$, $2^{*}:=\frac{2N}{N-2}$ is 
the Sobolev critical exponent, and $V:\mathbb{R}^n \to \mathbb{R}$ is a non-negative locally Lipschitz trapping potential in the sense that for each $x_0 \in \mathbb{R}^n$, there exists a neighborhood $B(x_0)$ of $x_0$ and a constant $L>0$ (depending on $B(x_0)$) such that
\begin{equation}\label{potential:1}
|V(x)-V(y)| \le L |x-y|, \quad \forall x,y \in B(x_0),
\end{equation}
and
\begin{equation}\label{potential:2}
V \geq 0 \quad \mbox{and } \ \lim_{|x| \to \infty}V(x)=+\infty.
\end{equation}

Throughout the whole paper, we shall always assume that $V$ satisfies \eqref{potential:1}, \eqref{potential:2}, and in addition, there exists a constant $C>0$ such that
\begin{equation}\label{potential:special}
- 2 V(x) \leq \nabla V(x) \cdot x \leq C V(x).
\end{equation}

In this paper, we will find solutions of equation \eqref{eq:main}, thus, the associated functional should be well defined. We introduce the energy space
\begin{equation}\label{space:main}
\mathcal{H}:=\left\{U\in H^{1}(\mathbb{R}^{N}): \int_{\R^{N}}\left(|\nabla U|^{2} +(V+1) U^{2} \right) dx < + \infty\right\},
\end{equation}
endowed with the norm
\begin{equation}\label{space:norm}
\|U\|_{\mathcal{H}}^{2}= \int_{\R^{N}} \left(|\nabla U|^{2} +(V+1) U^{2} \right) dx.
\end{equation}
Then, solutions of \eqref{eq:main} correspond to critical points of the energy functional 
$E: \mathcal{H}\rightarrow \mathbb{R}$, defined by
\begin{equation*}\label{func:main}
E(U)=\frac{1}{2}\int_{\mathbb{R}^{N}}|\nabla U|^{2}dx + \frac{1}{2}\int_{\mathbb{R}^{N}}V(x) U^{2}dx-\frac{1}{2^{*}}\int_{\mathbb{R}^{N}}|U|^{2^{*}}dx
\end{equation*}
constrained to the $L^{2}-$sphere
\begin{equation*}\label{manifold:rho}
\mathcal{M}_{\rho}=\left\{U\in \mathcal{H}: \int_{\mathbb{R}^{N}} U^{2}=\rho^{2}\right\},
\end{equation*}
where the number $\lambda$ playing the role of a Lagrange multiplier. 

To state our results we need the following definition.
\begin{definition}\label{def:GS}
We say that $U_{0}$ is a  normalized ground state of \eqref{eq:main} if $U_{0}$ is a critical point of $E$ constrained to $\mathcal{M}_\rho$, i.e. it solves \eqref{eq:main}, and
\begin{equation*}
E(U_{0})= \inf\{E(U): U \in \mathcal{M}_\rho, \ \nabla_{\mathcal{M}_\rho}E(U)=0\}.
\end{equation*}
\end{definition}

\begin{theorem}\label{thm:GS}
Let $N \geq 3$ and assume that \eqref{potential:1}, \eqref{potential:2} and \eqref{potential:special} hold. Then there exists $\rho^*>0$ such that, for every $0<\rho<\rho^*$,  \eqref{eq:main} has a positive solution, which corresponds to a local minimizer of $E$ on $\mathcal{M}_\rho$, and which is, in addition, a ground state.
\end{theorem}

\begin{theorem}\label{thm:MP}
Let $N \geq 3$ and assume that \eqref{potential:1}, \eqref{potential:2} and \eqref{potential:special} hold. Moreover, if $N \geq 6$, assume also
\begin{equation}\label{assump:Nge6}
V(0)=0.
\end{equation}
Then there exists $\rho^{*}>0$ such that, for every $0<\rho<\rho^{*}$, \eqref{eq:main} has a second positive solution, which is at a mountain pass level $E$ on $\mathcal{M}_\rho$.
\end{theorem}


\begin{remark}
Since $V$ is a locally Lipschitz potential with $V \ge 0$, the strong maximum principle implies that the non-negative solutions obtained in Lemma \ref{lem:local_min_attained} and \ref{lem:mon_minmax_attained} are in fact strictly positive.
\end{remark}

\begin{remark}
A typical example satisfying assumptions \eqref{potential:1}-\eqref{potential:special} is the polynomial-type potential
\[
V(x)=|x|^{\alpha},
\]
for any $\alpha > 0$. Moreover, the assumption \eqref{assump:Nge6} is also satisfied.
\end{remark}

\begin{remark}
We will see that the assumption \eqref{assump:Nge6} can be weakened. By Lemma \ref{lem:es_case3}, we obtain that 
\begin{equation}\label{assump:notoptimal}
\|V\|_{L^{\infty}(B_{R}(0))}< \frac{2s^2}{s+1} \lambda_{\mu},
\end{equation}
where $s \to 1^+$ defined in Lemma \ref{lem:es_case3}, $\lambda_{\mu}>0$ denotes the Lagrange multiplier associated with the ground state $u_{\mu}$ introduced in Section \ref{sec:GS}, $R>0$ can be chosen arbitrarily small.
\end{remark}

\begin{remark}\label{rmk:diff_estimate}
The assumption \eqref{assump:notoptimal} may not be optimal for the proof of the mountain-pass level. We can employ alternative estimation methods, as in \cite{MR4847285}. By \cite[Lemma 3.7.]{MR4847285}, we obtain a different estimate, for $0<\rho<\rho^{**}\leq \rho^{*}$, $N \geq 4$ and
\[
\|V\|_{L^{\infty}(B_{R}(0))}< \frac{1}{N} \lambda_{1}+h(\rho),
\]
where $\lambda_1$ is the first eigenvalue defined in \eqref{eq:eigen1}, $h(\rho) \to 0$ as $\rho \to 0$. However, this approach fails in the case $N=3$. Since in $N=3$, it requires the additional assumption
\[
\lambda_1(B_{R}(0))<\frac{4}{3}\lambda_1,
\]
which in turn implies that $R$ must be sufficiently large. This would require $V$ to be small on a large ball, which is not a desirable assumption.
\end{remark}

\begin{remark}
With the help of the left-hand side of assumption \eqref{potential:special}, we can determine the sign of $\lambda$. Moreover, an interesting question to consider is what would happen in the absence of this condition.

In this case, the local minimizer $u_{\mu}$ obtained in Theorem \ref{thm:GS} may not be a ground state. More precisely, Lemma \ref{lem:loc_GS} may no longer hold. It follows that Lemma \ref{lem:es_case3} fails. However, the energy level can still be estimated using the method described in Remark \ref{rmk:diff_estimate}. By applying Struwe's monotonicity trick (a nontrivial adaptation of \cite{JJ_BPS}), we obtain the desired result on a set of positive measure.

\end{remark}

\begin{remark}
The main difficulties in our analysis are the following. First, the boundedness of the Palais-Smale sequences is nontrivial due to the normalization constraint. Second, the compactness of the Palais-Smale sequences is delicate because of the presence of the Sobolev critical exponent $2^*$.  

To overcome these difficulties, we employ two main tools. For the first issue, we use a scaling technique introduced in \cite{MR1430506}. For the second, we follow the idea in  \cite[Ch.~III, Thm.~3.1]{Struwebook} and develop the following approach to analyze the loss of compactness: if a bounded Palais-Smale sequence does not converge strongly, the corresponding critical levels jumps of a fixed quantity.
\end{remark}

The paper is organized as follows: In Section \ref{sec:prel}, we first introduce a change of variables that moves the parameter from the constraint to the equation (and to the energy). We then present several preliminary results, focusing in particular on the estimate of $\lambda$ and the compactness properties of Palais-Smale sequences; Section \ref{sec:MP_geo} is devoted to the description of the mountain pass geometry and the introduction of two distinct candidate critical levels of $E$; finally, we prove our main results in Sections \ref{sec:GS} (Theorem \ref{thm:GS}) and \ref{sec:mp} (Theorems \ref{thm:MP}).

\section{Notation and preliminary results}\label{sec:prel}

For convenience of calculation we apply the transformation
\begin{equation}\label{trans}
u= \frac{1}{\rho} U,\qquad \ \ \ \ \ \mu= \rho^{2^{*}-2}>0,
\end{equation}
to convert problem \eqref{eq:main} into the following one, which also incorporates the sign condition:
\begin{equation}\label{eq:main1}
\begin{cases}
-\Delta u +V(x)u = \lambda u + \mu |u|^{2^{*}-2}u  & \text{in }\mathbb{R}^{N},\smallskip\\
u\geq 0,\quad \int_{\mathbb{R}^{N}} u^2\,dx = 1..
\end{cases}
\end{equation}
Thus, the solutions of \eqref{eq:main1} correspond to critical points of the energy functional $E_{\mu}: \mathcal{H}\rightarrow \mathbb{R}$, defined by
\begin{equation}\label{func:main1}
E_{\mu}(u)=\frac{1}{2}\int_{\mathbb{R}^{N}}|\nabla u|^{2}dx + \frac{1}{2}\int_{\mathbb{R}^{N}}V(x) u^{2}dx-\frac{\mu}{2^{*}}\int_{\mathbb{R}^{N}}|u|^{2^{*}}dx
\end{equation}
constrained to the $L^{2}$-sphere
\begin{equation*}\label{manifold:1}
\mathcal{M}:=\mathcal{M}_1=\left\{u\in \mathcal{H}: \int_{\mathbb{R}^{N}} u^{2}=1\right\}.
\end{equation*}

Since $U=\rho u$ and $u \in \mathcal{M}$, we have $U\in \mathcal{M}_{\rho}$. It follows that, once the existence of a local minimizer, a ground state, and a mountain pass solution for $E_\mu$ over $\mathcal{M}$ is established, Theorems \ref{thm:GS} and \ref{thm:MP} can be derived by the change of variable in \eqref{trans}.

From the definition of $ \mathcal{H}$, we have the compact embedding
\begin{equation}\label{embed:trap}
\mathcal{H}\hookrightarrow L^{p}(\R^{N}) \  \quad \mbox{for } 2 \leq p <2^{*}.
\end{equation}

By \eqref{potential:1} and \eqref{potential:2}, we introduce the the first eigenvalue $\lambda_{1}>0$ and the first eigenfunction $\varphi_{1}>0$, $\|\varphi_{1}\|^{2}_{2}=1$ with $\varphi_{1} \in \mathcal{H}$, $\varphi_{1} \in C^{2,\beta}(\R^{N})$ for every $0<\beta<1$, of the problem
\begin{equation}\label{eq:eigen}
-\Delta \varphi_{1}+V(x) \varphi_{1}= \lambda_{1} \varphi_{1},
\end{equation}
(see \cite[end of Section 1]{ntvDCDS}).

First, we have the following classical Pohozaev identity.
\begin{lemma}\label{lem:Pohozaev}
If $u$ solves
\begin{equation*}
\begin{cases}
-\Delta u +V(x)u = \lambda u + \mu |u|^{2^{*}-2}u  & \text{in }\mathbb{R}^{N},\smallskip\\
u\geq 0,\quad \int_{\mathbb{R}^{N}} u^2\,dx = 1.
\end{cases}
\end{equation*}
We have $u$ satisfies the following Pohozaev identity
\begin{equation}\label{eq:Pohozaev}
\int_{\R^N} |\nabla u|^2 \,dx- \frac{1}{2} \int_{\R^N} \nabla V(x) \cdot x u^2 \, dx = \mu \int_{\R^N} |u|^{2^{*}} dx
\end{equation}
\end{lemma}

We then establish an estimate for the value of $\lambda$.

\begin{lemma}\label{lem:max_Princ}
Let $(u,\lambda)\in \mathcal{H}\times \R$ be a  solution of \eqref{eq:main1}, where $u$ is nontrivial and non-negative.
Then 
\[
0\leq \lambda<\lambda_1.
\]
\end{lemma}

\begin{proof}
By taking $\varphi_{1}$ and $u$ as test functions in \eqref{eq:main1} and \eqref{eq:eigen}, respectively, we deduce
\begin{align*}
     \int_{\R^{N}} \nabla u \cdot\nabla \varphi_{1}dx +   \int_{\R^{N}} V(x) u  \varphi_{1}dx&=  \lambda\int_{\R^{N}} u  \varphi_{1}dx +\mu\int_{\R^{N}}  |u|^{2^{*}-2}u  \varphi_{1}dx,\\
     \int_{\R^{N}} \nabla u \cdot\nabla \varphi_{1}dx +   \int_{\R^{N}} V(x) u  \varphi_{1}dx&= \lambda_{1}\int_{\R^{N}}  u  \varphi_{1} dx, 
\end{align*}
which means that
\begin{equation*}
  (\lambda_{1}-\lambda) \int_{\R^{N}}  u  \psi_{1} dx = \mu\int_{\R^{N}}  |u|^{2^{*}-2}u  \psi_{1}dx.
\end{equation*}
Recalling that both $u$ and $\psi_{1}$ are strictly positive in $\R^N$, and $\mu>0$, we obtain that $\lambda <  \lambda_{1}$. 

On the other hand, by \eqref{eq:Pohozaev} and \eqref{eq:main1}, we have
\[
\lambda \int_{\R^N} u^2 dx= \frac{1}{2} \int_{\R^N} \nabla V(x) \cdot x u^2 \, dx +\int_{\R^N} V(x) u^{2} dx.
\]
The result is straightforward under assumption assumption \eqref{potential:special} and $\|u\|_2 =1$.
\end{proof}

As an immediate consequence of Lemma \ref{lem:max_Princ}, it follows that $E_{\mu}$ is bounded from below.
\begin{lemma}\label{lem:loc_GS}
Let $u$ be a solution of \eqref{eq:main1}. Then
\[
E_{\mu}(u) \geq \frac{1}{N}\left(\|u\|^{2}_{\mathcal{H}}-1\right).
\]
\end{lemma}

\begin{proof}
By Lemma \ref{lem:max_Princ}, we know that $\lambda\geq 0$. Since $u$ solves \eqref{eq:main1}, we have
\begin{equation*}
\begin{aligned}
E_{\mu}(u)&=\frac{1}{2}\|u\|^{2}_{\mathcal{H}}-\frac{1}{2}-\frac{\mu}{2^{*}}\|u\|^{2^{*}}_{2^{*}}\\
&=\frac{1}{2}\|u\|^{2}_{\mathcal{H}}-\frac{1}{2}- \frac{1}{2^{*}}\left(\|u\|^{2}_{\mathcal{H}}-1\right)+\frac{1}{2^{*}} \lambda\\
& =\frac{1}{N}\left(\|u\|^{2}_{\mathcal{H}}-1\right) +\frac{1}{2^{*}} \lambda\geq \frac{1}{N}\left(\|u\|^{2}_{\mathcal{H}}-1\right),
\end{aligned}
\end{equation*}
and the lemma follows.
\end{proof}

From \eqref{potential:special} and Gr\"onwall's inequality, we have the following lemma concerning the potential item.
\begin{lemma}\label{lem:pot_scal}
For every $h>0$, there exists $C(h)>0$ such that
\begin{equation*}\label{eq:pot_scal}
\int_{\R^N}V(hx) u^2 dx < C(h) \int_{\R^N}V(x) u^2 dx.
\end{equation*}
\end{lemma}

\begin{proof}
Recall \eqref{potential:special}, which yields
\begin{equation*}\label{ineq:assum}
-2 V(x) \leq \nabla V(x) \cdot x \leq C V(x).
\end{equation*}
First, let $h \geq 1$. Applying Gr\"onwall's inequality to the map $h \mapsto V(hx)$, we obtain
\[
\frac{\partial}{\partial h} \left( h^{-C} V(hx)\right) =  -C  h^{-C-1} V(hx)+h^{-C} \nabla V(hx) \cdot x\leq 0,
\]
which implies that
\[
V(hx) \leq h^{C} V(x).
\]
Similarly, for $0 < h < 1$, we have
\[
V(hx) \leq h^{-2} V(x),
\]
which completes the proof.
\end{proof}

Moreover, we recall the following concentration-compactness result from 
\cite{MR1632171,Lewinlecture}.
\begin{lemma}\label{Concom}
Let $(u_{n})_{n}$ be a bounded sequence in $D^{1,2}(\mathbb{R}^{N})$, and define
\begin{equation*}
  m(u) = \sup \left\{\int_{\mathbb{R}^{N}} |u|^{2^{*}} : \begin{aligned}
  \text{there exists }(n_k)_k &\subset\N,\  
  (\alpha_k)_k\subset\R,\  
  (x_k)_k\subset\R^N,\text{ with }\\
  &\alpha_{k}^{-N/2^{*}}u_{n_{k}} \left(\frac{\cdot-x_{k}}{\alpha_{k}}\right) \rightharpoonup u
  \end{aligned}\right\}.
\end{equation*}
Then, $m(u)=0$ if and only if $u_{n}\rightarrow0$ strongly in $L^{2^{*}}(\mathbb{R}^{N})$.
\end{lemma}

To proceed with the analysis of the convergence of Palais-Smale sequences, we turn to the following result, which is a variant of \cite[Ch. III,
Thm. 3.1]{Struwebook}.

\begin{lemma}\label{lem:split}
Let $(u_{n},\lambda_{n})_{n}$ be a bounded Palais-Smale sequence for $E_{\mu}$ on $\mathcal{M}$ and assume that
\[
\|u^{-}_{n}\|_{\mathcal{H}} \to 0.
\]
Then, up to subsequences, there exist $u^{0} \in \mathcal{H}$, $u^{0} \geq 0$ and $\lambda^{0} \in \R$ such that $u_{n}\rightharpoonup u^{0}$ weakly in $\mathcal{H}$, $\lambda_{n} \to \lambda^{0}>0$ and
\[
- \Delta u^{0} +V(x) u^{0} = \mu (u^{0})^{2^{*}-1}+ \lambda^{0} u^{0}.
\]
Moreover, only one of the two following holds:
\begin{enumerate}
\item $u_{n} \to u^{0}$ strongly in $\mathcal{H}$,
\item\label{eq:split_case2} $u_{n} \rightharpoonup u^{0}$ in $\mathcal{H}$ (but not strongly), and
\begin{equation}\label{energy:level_split}
E_\mu (u_{n}) \geq E_\mu (u^{0}) + \frac{1}{N}S^{N/2} \mu^{1-\frac{N}{2}}+o(1), \quad \mbox{as} \quad n \to + \infty.
\end{equation}
\end{enumerate}
\end{lemma}

\begin{proof}
Since $(u_{n})_{n}$ is bounded in $\mathcal{H}$ and $\|u_{n}^{-}\|_{\mathcal{H}} \to 0$, there exist $u^{0} \geq 0 $ such that $u_{n} \rightharpoonup 0$ weakly in $\mathcal{H}$. As recalled in \eqref{embed:trap}, due to the compact embedding, we have
\begin{equation*}\label{eq:L2embed}
u_{n} \to u^{0} \quad \mbox{in} \quad L^{2}(\R^{N}).
\end{equation*}
Setting $u_{1,n}=u_{n}-u^{0}$, we have
\[
u_{1,n} \rightharpoonup 0 \quad \mbox{ weakly   in } \mathcal{H} \qquad \mbox{and} \quad \mbox{ strongly in } L^{p}(\R^{N}) \quad p \in [2,2^{*}).
\]
By the Brezis-Lieb lemma, we have
\begin{equation*}
\begin{aligned}
\|\nabla u_{n}\|_{2}^{2}=\|\nabla u^{0}\|_{2}^{2}+\|\nabla u_{1,n}\|_{2}^{2}+o(1), \quad  \| u_{n}\|_{2^{*}}^{2^{*}}=\| u^{0}\|_{2^{*}}^{2^{*}}+\| u_{1,n}\|_{2^{*}}^{2^{*}}+o(1),
\end{aligned}
\end{equation*}
and
\[
\int_{\R^{N}} V(x)u_{n}^{2} dx=\int_{\R^{N}} V(x)(u^{0})^{2} dx+\int_{\R^{N}} V(x)u_{1,n}^{2} dx+o(1),
\]
which yield that
\begin{equation}\label{equation:function_split}
E_{\mu}(u_{n})= E_{\mu}(u^{0})+E_{\mu}(u_{1,n})+o(1).
\end{equation}
Now, by the defination of the energy space $\mathcal{H}$ and the norm (see \eqref{space:main} and \eqref{space:norm}), $(u_{1,n})_{n}$ satisfies
\begin{equation}\label{eq:u12eq}
\|u_{1,n}\|_{\mathcal{H}}^{2}=\mu \|u_{1,n}\|^{2^{*}}_{2^{*}}+o(1).
\end{equation}
Thus, if $u_{1,n} \to 0$ in $\mathcal{H}$, we are in case $1.$ If not, we we are left to show that case $2$ holds. We assume that there exists a suitable positive constant $\delta$ such that
$\|u_{1,n}\|_{\mathcal{H}} \geq \delta $. By \eqref{eq:u12eq}, we obtain that there exists a constant $\delta_1>0$ such that
\[
 \|u_{1,n}\|_{2^{*}}\geq \delta_1>0.
\]
By Lemma \ref{Concom} there exist sequences $(\alpha_{n})_n$, $(x_{n})_n$, and a non-trivial $u^{1}\in \mathcal{H}$ such that
\begin{equation*}\label{seq_dilation}
  \tilde{u}_{1,n}=\alpha_{n}^{-\frac{N-2}{2}}u_{1,n}\left(\frac{\cdot-x_{n}}{\alpha_{n}}\right)\rightharpoonup u^{1}
  \qquad\text{weakly in }\mathcal{H}.
\end{equation*}
Without loss of generality, after translating we may assume $x_n=0$; moreover, direct computations show that
\begin{equation}\label{norm_dilation}
  \|\nabla \tilde{u}_{1,n}\|^{2}_{2}= \|\nabla u_{1,n}\|^{2}_{2}, \ \ \ \| \tilde{u}_{1,n}\|^{2^{*}}_{2^{*}}= \| u_{1,n}\|^{2^{*}}_{2^{*}}, \ \ \ \mbox{and } \ \| \tilde{u}_{1,n}\|^{2}_{2}= \alpha_{n}^{2}\| u_{1,n}\|^{2}_{2}.
\end{equation}
Therefore, up to subsequences, we have two different cases. 

Case 1: $\alpha_n \not \to \infty$. From $u_{1,n} \in \mathcal{H}$, \eqref{embed:trap} and \eqref{norm_dilation}, we have
\[
\| \tilde{u}_{1,n}\|^{2}_{2}= \alpha_{n}^{2}\| u_{1,n}\|^{2}_{2} \to 0, \quad \mbox{as } n \to +\infty,
\]
which implies that $\tilde{u}_{1,n} \to 0$ strongly in $L^2(\R^N)$. Thus, we can deduce that $u^1=0$ a.e. in $\R^N$, a contradiction with the fact that $u^1$ is nontrivial.

Case 2: $\alpha_n \to \infty$. Without loss of generality, we may assume that $\alpha_n > 1$ for large $n$. For every $\psi (x)\in C^{\infty}_{0}(\R^N)$, we define 
\[
\varphi_n(y) = \psi(\alpha_n y).
\]
Hence, $\varphi_n(x) \in C^{\infty}_0(\mathbb{R}^N)$ and
\[
  \|\varphi_{n}\|_{2}^{2}= \alpha_{n}^{-N}\|\psi\|^{2}_{2}, \ \ \ \|\nabla \varphi_{n}\|^{2}_{2}=\alpha_{n}^{2-N}\|\nabla\psi\|^{2}_{2}.
\]
By \eqref{eq:u12eq} and $\varphi_n (x)\in C^{\infty}_{0}(\R^N)$, we have, as $n\rightarrow +\infty$,
\begin{equation*}
  \int_{\mathbb{R}^{N}}\nabla u_{1,n}\nabla \varphi_n \, dx - \mu \int_{\mathbb{R}^{N}} |u_{1,n}|^{2^{*}-2}u_{1,n}\varphi_n\, dx = \int_{\mathbb{R}^{N}} \left( \lambda^{0}-V(x)\right)u_{1,n}\varphi_n \, dx+o\left(\|\varphi_n \|_{H^{1}}\right).
\end{equation*}
From this, we obtain that
\begin{equation}\label{eq:tilde_psi_varphi}
\begin{aligned}
&\int_{\mathbb{R}^{N}}\nabla \tilde{u}_{1,n}(x) \nabla \psi(x)dx -\mu  \int_{\mathbb{R}^{N}}|\tilde{u}_{1,n}(x)|^{2^{*}-2}\tilde{u}_{1,n}(x) \psi(x)dx \\
=&\alpha_{n}^{\frac{N}{2}-1}\left[\int_{\mathbb{R}^{N}}\nabla u_{1,n}\nabla \varphi_n\, dx- \mu \int_{\mathbb{R}^{N}} |u_{1,n}|^{2^{*}-2}u_{1,n} \varphi_n\, dx\right]\\
=&\alpha_{n}^{\frac{N}{2}-1}\int_{\mathbb{R}^{N}} \left( \lambda^{0}-V(x)\right)u_{1,n}\varphi_n \, dx+o\left(\|\varphi_n\|_{H^{1}}\right).
\end{aligned}
\end{equation}
Since $V \in C^{0,1}_{loc}(\R^N)$, $\psi(x) \in C_0^\infty(\mathbb{R}^N)$, and $(u_{1,n})_n$ is a bounded sequence in $\mathcal{H}$, it follows from the Hölder inequality that there exists a constant $C_1>0$ such that
\begin{equation*}
\begin{aligned}
\int_{\R^N} V(x) u_{1,n}(x) \varphi_{n}(x) dx &\leq \left(\int_{\R^N} V(x)u^{2}_{1,n}(x)dx\right)^{\frac{1}{2}}\left(\int_{\R^N} V(x)\varphi^{2}_{n}(x)dx\right)^{\frac{1}{2}}\\
&= \alpha_n^{-\frac{N}{2}} \left(\int_{\R^N} V(x)u^{2}_{1,n}(x)dx\right)^{\frac{1}{2}}\left(\int_{\R^N} V\left(\frac{x}{\alpha_n}\right)\psi^{2}(x)dx\right)^{\frac{1}{2}}\\
&\leq  \alpha_n^{-\frac{N}{2}}  C_1.
\end{aligned}
\end{equation*}
Therefore, by \eqref{eq:tilde_psi_varphi}, we have
\begin{equation*}
\begin{aligned}
&\left|  \int_{\mathbb{R}^{N}} \nabla \tilde{u}_{1,n}\nabla \psi \,dx -\mu  \int_{\mathbb{R}^{N}}|\tilde{u}_{1,n}|^{2^{*}-2}\tilde{u}_{1,n} \psi \, dx \right| \\
&\leq \alpha_{n}^{\frac{N}{2}-1}\left[ \lambda^0 \|u_{1,n}\|_{2} \|\varphi_n\|_2  +  \left(\int_{\R^N} V(x)u^{2}_{1,n}(x)dx\right)^{\frac{1}{2}}\left(\int_{\R^N}V(x)\varphi^{2}_{n}(x)dx\right)^{\frac{1}{2}} +o(\|\varphi_{n}\|_{H^{1}})\right]\\
&\leq C \lambda^{0} \alpha_{n}^{-1} + C_1 \alpha_{n}^{-1} +o\left(\|\nabla \psi \|_2 + \alpha_n^{-1} \|\psi\|_2 \right) \to 0,
\end{aligned}
\end{equation*}
as $\alpha_n \to + \infty$. Thus, we obtain that $\tilde{u}_{1,n}$ satisfies
\begin{equation}\label{eq_split}
  \int_{\mathbb{R}^{N}}|\nabla \tilde{u}_{1,n}|^{2}\,dx -\mu  \int_{\mathbb{R}^{N}}|\tilde{u}_{1,n}|^{2^{*}}\,dx=o(1),\qquad
  \text{ and }-\Delta u^1=\mu |u^1|^{2^{*}-2}u^1 .
\end{equation}
Hence, combining \eqref{eq_split} and the Sobolev inequality, we have
\begin{equation*}
\begin{aligned}
  \frac12 \|\nabla \tilde{u}_{1,n}\|_{2}^{2} -\frac{\mu}{2^*}\|\tilde{u}_{1,n}\|_{2^{*}}^{2^{*}}&=\left(\frac{1}{2}-\frac{1}{2^{*}}\right)\|\nabla \tilde{u}_{1,n}\|_{2}^{2}+o(1)\\
 & =\frac{1}{N}\|\nabla \tilde{u}_{1,n}\|_{2}^{2}+o(1) \geq\frac{1}{N}\|\nabla u^{1}\|_{2}^{2}+o(1) \geq \frac{1}{N}S^{\frac{N}{2}}\mu^{1-\frac{N}{2}}+o(1).
  \end{aligned}
\end{equation*}
Then, by $V \geq 0$, we have
\begin{equation*}
 E_{\mu}(u_{1,n})=  \frac12 \|\nabla \tilde{u}_{1,n}\|_{2}^{2} -\frac{\mu}{2^*}\|\tilde{u}_{1,n}\|_{2^{*}}^{2^{*}}+\frac{1}{2}\int V(x) u^2_{1,n} dx \geq \frac{1}{N}S^{\frac{N}{2}}\mu^{1-\frac{N}{2}}+o(1),
\end{equation*}
which, together with \eqref{equation:function_split}, yields \eqref{energy:level_split}.
\end{proof}

\section{Mountain pass geometry}\label{sec:MP_geo}
In order to investigate the geometric structure of $E_{\mu}$, we define the sets
\begin{equation}\label{space:geometry}
\mathcal{B}_{\alpha}=\left\{ u \in \mathcal{M}: \|u\|^{2}_{\mathcal{H}} <\alpha \right\} \quad \mbox{and} \quad \mathcal{U}_{\alpha}=\left\{ u \in \mathcal{M}: \alpha- \eps \leq \|u\|^{2}_{\mathcal{H}} \leq \alpha \right\},
\end{equation}
where $\eps>0$ is a sufficiently small fixed constant. Recall \eqref{eq:eigen}, we have the first eigenvalue $\lambda_{1}>0$ and the first eigenfunction $\varphi_{1}>0$, $\|\varphi_{1}\|^{2}_{2}=1$ with $\varphi_{1} \in \mathcal{M}$, $\varphi_{1} \in C^{2,\beta}(\R^{N})$ for every $0<\beta<1$ of the problem
\begin{equation}\label{eq:eigen1}
-\Delta \varphi_{1}+V(x) \varphi_{1}= \lambda_{1} \varphi_{1},
\end{equation}
which implies that $\mathcal{B}_{\alpha}$ is non-empty if and only if $\alpha \geq \lambda_{1}+1$.

Then, we find the $E_{\mu}$ possesses the following mountain pass structure.

\begin{lemma}\label{lem:geometry}
Assume that
\begin{equation}\label{eq:muless}
\mu<\left(  \frac{S^{N/2}}{N(\lambda_{1}+1)}  \right)^{\frac{2}{N-2}}
\end{equation}
Let us define
\begin{equation}\label{eq:baralpha}
\bar{\alpha}= S^{N/2} \mu^{1-\frac{N}{2}}+1.
\end{equation}
Then
\begin{equation*}\label{eq:geometry}
0< \inf_{\mathcal{B}_{\bar{\alpha}}}E_{\mu}\leq E(\varphi_{1})< E^{*},
\end{equation*}
where
\[
E^{*}= \min \left \{ E(u): u \in \mathcal{U}_{\bar{\alpha}} \right\}.
\]

\end{lemma}

\begin{proof}
From \eqref{eq:eigen1}, we have
\begin{equation}\label{eq:varphi_lambda}
E_{\mu}(\varphi_{1})= \frac{1}{2}\left(\|\varphi_{1}\|^{2}_{\mathcal{H}} -1\right)-\frac{\mu}{2^{*}}\|\varphi_{1}\|^{2^{*}}_{2^{*}}= \frac{1}{2} \lambda_{1}  -\frac{\mu}{2^{*}}\|\varphi_{1}\|^{2^{*}}_{2^{*}} \leq \frac{1}{2} \lambda_{1}.
\end{equation}
Then, combining \eqref{eq:muless}, \eqref{eq:baralpha} and \eqref{eq:varphi_lambda}, we have
\begin{equation*}
\begin{aligned}
E^{*} \geq E_{\mu}(u) &\geq \frac{1}{2} \left(\|u\|^{2}_{\mathcal{H}}-1\right) - \frac{\mu}{2^{*}} \|u\|^{2^{*}}_{2^{*}}\\
& \geq \frac{1}{2} \left(\|u\|^{2}_{\mathcal{H}}-1\right) -\frac{\mu}{2^{*}S^{2^{*}/2}} \left(\|u\|^{2}_{\mathcal{H}}-1 \right)^{\frac{2^{*}}{2}}\\
& \geq \left(\frac{1}{N}-\frac{\eps}{2}\right)S^{N/2} \mu^{1-\frac{N}{2}} -\frac{1}{2}> \frac{1}{2} \lambda_{1},
\end{aligned}
\end{equation*}
which yields that
\[
E^{*}> \frac{1}{2} \lambda_{1} \geq E(\varphi_{1}) \geq  \inf_{\mathcal{B}_{\bar{\alpha}}}E_{\mu} >0. \qedhere
\]
\end{proof}

Motivated by the previous lemma, we introduce two energy levels, depending on $\mu$: one associated with a local minimizer, and the other with a mountain pass. Both levels are regarded as candidate critical points for $E_\mu$ (as defined in \eqref{func:main1}), whose actual criticality will be discussed in the following sections.

First, we introduce the energy level associated with a candidate local minimum as
\begin{equation}\label{definition:local_minimizer}
m_{\mu}= \inf_{\mathcal{B}_{\bar{\alpha}}} E_{\mu}.
\end{equation}
By Lemma \ref{lem:geometry}, we have
\begin{equation}\label{LM_bal}
m_{\mu} < \frac{1}{N} S^{N/2} \mu^{1-\frac{N}{2}}.
\end{equation}
On the other hand, after taking into account the candidate local minimum, we turn to a second candidate critical value of mountain pass type. By Lemma \ref{lem:geometry}, any $w \in \mathcal{U}_{\alpha}$ satisfies
\begin{equation}\label{MPG_bal}
 E_{\mu}(w)>\frac{1}{N} S^{N/2} \mu^{1-\frac{N}{2}}
\end{equation}
which provides the necessary energy separation from the local minimum, ensuring that any continuous path starting in the interior of the ball $\mathcal{B}_\alpha$ and leaving its closure (i.e., entering the exterior region $\mathcal{H} \setminus \mathcal{B}_\alpha$) must pass through a point on the sphere $\mathcal{U}_\alpha$ with energy strictly exceeding $\frac{1}{N} S^{N/2} \mu^{1-\frac{N}{2}}$. Consequently, this condition provides the mountain pass geometry required to define the corresponding critical value.

To this end, we need to select two functions $w_1, w_2$ in $\mathcal{M}$, serving as endpoints for the mountain pass, satisfying the following properties:

\begin{equation}\label{MPG_in_bal}
\|w_1\|_{\mathcal{H}}^{2}<\alpha, \ \ \qquad \ \ E_{\mu}(w_{1})<\frac{1}{N} S^{N/2} \mu^{1-\frac{N}{2}},
\end{equation}
and
\begin{equation}\label{MPG_out_bal}
\|w_2\|_{\mathcal{H}}^{2}>\alpha, \ \ \qquad \ \ E_{\mu}(w_{2})\leq 0.
\end{equation}

Under the assumption \eqref{eq:muless}, it follows from Lemma \ref{lem:geometry} that we may choose $w_1=\varphi_1$, and consequently, we have
\begin{equation}\label{MPG_in_bal_1}
\|\varphi_1\|_{\mathcal{H}}^{2}= \lambda_1<\alpha, \ \ \qquad \ \ E_{\mu}(\varphi_{1})<\frac{1}{N} S^{N/2} \mu^{1-\frac{N}{2}}.
\end{equation}
where $\lambda_1$ and $\varphi_1$ defined in \eqref{eq:eigen1}. Hence, \eqref{MPG_in_bal} is established.

On the other hand, condition \eqref{MPG_out_bal} can be easily realized. For every $u \in \mathcal{M}$ and $h >0$, we define the scaling
\begin{equation}\label{eq:scaling}
u_h(x) := h \star u(x)=h^{\frac{N}{2}} u(h x),
\end{equation}
so that by Lemma \ref{lem:pot_scal}, the functional
\begin{equation}\label{func:scaling}
E_{\mu}(u_{h})=\frac{h^{2}}{2}\int_{\mathbb{R}^{N}}|\nabla u|^{2}dx + \frac{1}{2}\int_{\mathbb{R}^{N}}V\left(\frac{x}{h}\right) u^{2}dx-\frac{\mu h^{2^{*}}}{2^{*}}\int_{\mathbb{R}^{N}}|u|^{2^{*}}dx
\end{equation}
is well defined. Then, the following lemma ensures that the scaled function remains in the constraint $\mathcal{M}$ and can serve as the right endpoint of the mountain pass (see \eqref{MPG_out_bal}). 

\begin{lemma}\label{lem:scaling_MP}
Let $u \in \mathcal{M}$ be fixed. For every $h >0$, we have $u_h \in \mathcal{M}$ and
\begin{equation*}
\|\nabla u_h\|_2 \to + \infty \qquad \mbox{and } \qquad E_{\mu}(u_{h}) \to -\infty \qquad \mbox{as } \quad h \to +\infty.
\end{equation*}
\end{lemma}

\begin{proof}
It follows from \eqref{eq:scaling} that $\|u_h\|_2^2=\|u\|_2^2$. Then, by $u \in \mathcal{H}$ and Lemma \ref{lem:pot_scal}, there exist constants $C(h)>0$ such that, for every $h > 0$, we have
\[
\int_{\R^N}V(hx) u^2 dx < C(h) \int_{\R^N}V(x) u^2 dx<\infty ,
\]
which yields that $u_h \in \mathcal{M}$.

Moreover, from \eqref{func:scaling}, we have
\begin{equation*}
E_{\mu}(u)=\frac{h^{2}}{2}\int_{\mathbb{R}^{N}}|\nabla u|^{2}dx + \frac{1}{2}\int_{\mathbb{R}^{N}}V\left(\frac{x}{h}\right) u^{2}dx-\frac{\mu h^{2^{*}}}{2^{*}}\int_{\mathbb{R}^{N}}|u|^{2^{*}}dx\to -\infty,
\end{equation*}
as $h \to + \infty$. This concludes the proof.
\end{proof}

Now, let $w_{0}=\varphi_1$ be chosen as the left endpoint. Starting from this point, the scaling \eqref{eq:scaling} together with Lemma \ref{lem:scaling_MP} allows us to find $h_1>0$ such that
\begin{equation}\label{eq:left_right}
\|w_{h_{1}}\|_{\mathcal{H}}^{2}>\alpha, \ \ \qquad \ \ E_{\mu}(w_{h_{1}})\leq 0,
\end{equation}
which satisfies \eqref{MPG_out_bal}. Moreover, our path passes through the set $\mathcal{U}_{\alpha}$(see the definition in \eqref{space:geometry}), thus satisfying \eqref{MPG_bal}. Then, we define in a standard way the mountain pass value
\begin{equation}\label{level:mp}
c_{\mu}:=\inf\limits_{\gamma\in\Gamma} \sup\limits_{t\in[0,1]}E_{\mu}(\gamma(t)),
\qquad\text{where }
\Gamma:=\{\gamma\in C([0,1], \mathcal{M}) : \gamma(0)=w_{0}, \gamma(1)=w{_{h_{1}}}\}.
\end{equation}

From \eqref{MPG_bal}, we immediately deduce that any path in $\Gamma$ must pass through the set $\mathcal{U}_{\alpha}$, and, moreover, by \eqref{LM_bal} and \eqref{MPG_bal}, we obtain
\begin{equation*}
 m_{\mu} < \frac{1}{N} S^{N/2} \mu^{1-\frac{N}{2}} < c_\mu,
\end{equation*}
which yields, in particular, that $c_\mu \neq m_\mu$; hence, if we demonstrate that both are indeed critical levels, problem \eqref{eq:main} admits two distinct solutions.

\section{Ground state solution (proof of Theorem \ref{thm:GS})}\label{sec:GS}

This section is devoted to the proof of Theorem \ref{thm:GS}. Throughout the section, we assume that assumptions \eqref{potential:1}-\eqref{potential:special} are satisfied. We first aim to prove the existence of the local minimizer. Under Lemma \ref{lem:loc_GS}, we then show that the local minimizer is the ground state, in the sense of Definition \ref{def:GS}.

First, Lemma \ref{lem:geometry} ensures that $m_\mu$ can be regarded as a candidate local minimum (see \eqref{definition:local_minimizer}). The following lemma provides an estimate for the energy level $m_{\mu}$.
\begin{lemma}\label{energylevel_localmin}
Under the assumptions of Lemma \ref{lem:geometry}, we have
\begin{equation}\label{eq:estimate_local}
\frac{1}{N} \lambda_1 \leq m_{\mu} < \frac{1}{N} S^{N/2} \mu^{1-\frac{N}{2}}.
\end{equation}
\end{lemma}

\begin{proof}
As a direct consequence of Lemma \ref{lem:geometry}, we obtain
\begin{equation*}
m_{\mu} \leq E(\varphi_{1}) < \frac{1}{N} S^{N/2} \mu^{1-\frac{N}{2}}.
\end{equation*}
By Lemma \ref{lem:max_Princ}, we have $\lambda \geq 0$. Then, combining \eqref{eq:main1} and \eqref{func:main1}, we obtain
\begin{equation*}
\begin{aligned}
m_{\mu} \geq E_{\mu}(u)&=\frac{1}{2}\int_{\mathbb{R}^{N}}|\nabla u|^{2}dx + \frac{1}{2}\int_{\mathbb{R}^{N}}V(x) u^{2}dx-\frac{\mu}{2^{*}}\int_{\mathbb{R}^{N}}|u|^{2^{*}}dx\\
& = \frac{1}{N} \left(\|u\|_{\mathcal{H}}^{2}-1\right)+ \frac{\lambda}{2^*} \|u\|_{2}^{2} \geq \frac{1}{N} \left(\|u\|_{\mathcal{H}}^{2} -1\right)\geq \frac{1}{N} \lambda_1.
\end{aligned}
\end{equation*}
This completes the proof.
\end{proof}
Combining this lemma with Lemma \ref{lem:split}, we naturally obtain the conclusion that $m_\mu$ is indeed a local minimum.

\begin{lemma}\label{lem:local_min_attained}
Assume that the hypotheses of Theorem~\ref{thm:GS} hold. Then, for every $0 < \mu < \mu^{*}$, the infimum $m_\mu$ is achieved, that is, there exists $u_\mu \in \mathcal{H}$ such that
\[
E_\mu(u_\mu) = m_\mu.
\]
In particular, $u_\mu$ is a local minimizer of $E_\mu$.
\end{lemma}

\begin{proof}
By Lemma \ref{lem:geometry}, which provides the Mountain pass geometric structure, we can construct a minimizing sequence $(v_n)_{n} \subset B_\alpha$ for $m_\mu$. 
Since both $E_\mu$ and $B_\alpha$ are even, we may assume without loss of generality that $v_n \geq 0$ for all $n$. 
Moreover, Lemma \ref{lem:geometry} ensures that the gradients of $v_n$ are uniformly bounded, that is,
\[
\|\nabla v_n\|_2^2 \leq \bar{\alpha}= S^{N/2} \mu^{1-\frac{N}{2}} \quad \text{for all } n.
\] 
Then, applying Ekeland's variational principle, we can construct a Palais-Smale sequence $(u_n)_{n} \subset B_\alpha$ for $E_\mu$ at the level $m_\mu$ such that 
\[
\|u_n - v_n\|_{\mathcal{H}} \to 0 \quad \text{as } n \to +\infty,
\]
and consequently,
\[
E_\mu(u_n) \to m_\mu, \quad \nabla_{\mathcal{M}} E_\mu(u_n) \to 0 \qquad \|u_n^-\|_{\mathcal{H}} \to 0.
\]
Moreover, since $(u_n)_{n}$ is bounded in $\mathcal{H}$, the associated Lagrange multipliers
\[
\lambda_n := \int_{\mathbb{R}^N} |\nabla u_n|^2 + V(x) u_n^2 \, dx - \mu \int_{\mathbb{R}^N} |u_n|^{2^*} \, dx
\]
are also bounded. Hence, the assumptions of Lemma \ref{lem:split}  are satisfied, and we can apply it. Since $(u_n)_{n}$ is bounded in $\mathcal{H}$, there exists $u^0 \in \mathcal{H}$ such that, up to a subsequence,
\[
u_n \rightharpoonup u^0 \text{ in } \mathcal{H}, \quad u_n \to u^0 \text{ strongly in } L^2(\R^N), \quad \text{and } u_n \to u^0 \text{ a.e. } \R^{N}. 
\]

To conclude the proof, we show that $u_n \to u^0$ strongly in $\mathcal{H}$, which implies that $m_\mu$ is attained. We argue by contradiction, assuming that $u_n \rightharpoonup u^0$ in $\mathcal{H}$ but that the convergence is not strong. Then, by \eqref{energy:level_split} and \eqref{eq:estimate_local}, we have
\[
E_\mu (u^{0})+\frac{1}{N}S^{N/2} \mu^{1-\frac{N}{2}}\leq m_\mu <   \frac{1}{N}S^{N/2} \mu^{1-\frac{N}{2}},
\]
which implies that $E_{\mu}(u^{0})<0$. However, since $u_n \to u^0$ strongly in $L^2(\mathbb{R}^N)$, we have $\|u^0\|_2 = 1$, which implies that $u^0$ is nontrivial. 
In view of the definition of $m_\mu$, it follows that
\[
E_\mu(u^0) \ge m_\mu \ge \frac{1}{N} \lambda_1 > 0,
\]
which is a contradiction. Hence, $u_n \to u^0$ strongly in $\mathcal{H}$, so that $u_\mu$ attains the local minimum of $E_\mu$.
\end{proof}

From Lemmas \ref{lem:loc_GS} and \ref{lem:geometry}, it follows that if $u \notin \mathcal{B}_{\bar{\alpha}}$ is a solution of \eqref{eq:main1}, then
\[
E_{\mu}(u) \geq \frac{1}{N} S^{N/2} \mu^{1-\frac{N}{2}} > m_{\mu}.
\]
This implies that the local minimizer $u_{\mu} \in \mathcal{B}_{\bar{\alpha}}$ is a ground state.

\section{Mountain pass solution (proof of Theorems \ref{thm:MP})}\label{sec:mp}

In this section, we prove Theorem \ref{thm:MP}. Throughout this section, we assume $ N \geq 3$ and assumptions \eqref{potential:1}-\eqref{potential:special} are satisfied. In the case $N \geq 6$, all other assumptions remain the same, but we consider the potential satisfies \eqref{assump:Nge6}. In this case, under this assumption and $V$ is locally Lipschitz potential, the potential can be suitably small on a small ball centered at the origin, which implies
\begin{equation}\label{eq:ass_MPsec}
\|V\|_{L^{\infty}(B_{R}(0))}< L R,
\end{equation}
where $R$ is a positive constant. The main strategy for proving the theorem is as follows, partly following the approach in \cite{MR4476243,verzini2025normalizedsolutionsnonlinearschrodinger,MR4433054}.

To begin with, by the mountain pass geometry (see Lemma \ref{lem:geometry}) and Lemma \ref{lem:scaling_MP}, we can construct an associated Palais-Smale sequence using the classical method introduced by Jeanjean in \cite{MR1430506}.

Then, by a theorem from Ghoussoub’s book \cite[Thm. 4.1]{MR1251958} and the proof strategy from \cite[Proposition 3.4.]{MR4443784}, we construct a bounded positive Palais–Smale sequence.

Finally, by combining Lemma \ref{lem:split} with the estimates on the energy level, we recover compactness and thereby obtain the desired conclusion.

At the first step, we define
\begin{equation}\label{def:E_tilde}
  \tilde{E}_{\mu}(u,h):=E_{\mu}(e^{\frac{N}{2}h}u(e^{h}x)), \ \ \ \ \mbox{for all }(u,h)\in \mathcal{M}\times \R,
\end{equation}
\begin{equation}\label{def:path_tilde}
  \tilde{\Gamma}:=\{\tilde{\gamma}\in C([0,1], \mathcal{M}\times \R):\tilde{\gamma}(0)=(w_{0},0),\tilde{\gamma}(1)=(w_{h_{1}},0)\}
\end{equation}
and
\begin{equation}\label{def:MP_tilde}
  \tilde{c}_{\mu}:=\inf_{\tilde{\gamma} \in\tilde{\Gamma}}\sup_{t\in[0,1]}\tilde{E}_{\mu}(\tilde{\gamma}(t))
\end{equation}
(recall the definition of $w_{0}=\varphi_1$ in \eqref{MPG_in_bal_1} and $w_{h_1}$ in \eqref{eq:left_right}).

\begin{lemma}[{\cite[Lemma 5.1.]{verzini2025normalizedsolutionsnonlinearschrodinger}}]\label{ener_compare}
We have
\begin{equation*}
  \tilde{c}_{\mu}=c_{\mu}. 
\end{equation*}
\end{lemma}

In order to estimate the energy associated with the mountain pass, we consider the following classical function. Let $N \ge 3$. We denote by $U_\varepsilon$ the Aubin-Talenti function centered at the origin:
\[
U_\varepsilon(x) :=  \frac{\varepsilon^{(N-2)/2}}{(\varepsilon^2 + |x|^2)^{(N-2)/2}}, \quad x \in \mathbb{R}^N,
\]
where $\varepsilon>0$ is a sufficiently small fixed constant. Let $R>0$ and let $\eta \in C^{\infty}_{0}(\R^N)$, $0 \leq \eta \leq 1$, be a cut-off function satisfying
\[
\eta = 1 \quad \text{in a neighborhood of } 0, \qquad 
\eta = 0 \quad \text{in } B_R^c(0).
\]
We define

\begin{equation}\label{eq:def_u_eps}
u_{\varepsilon}(x)=\frac{\eta[N(N-2)\varepsilon^{2}]^{^{\frac{N-2}{4}}}}{[\varepsilon^{2}+ |x|^{2}]^{\frac{N-2}{2}}}.
\end{equation}

Using the estimates provided by Struwe \cite[page 179]{Struwebook} (or those by
Brezis and Nirenberg, \cite[eqs. (1.13), (1.29)]{BrezisNirenberg}), we have
\begin{equation}\label{est:struwe}
    \begin{aligned}\displaystyle
\|\nabla u_{\varepsilon}\|_{2}^{2}&=S^{N/2}+O(\varepsilon^{N-2}),\\
\|u_{\varepsilon}\|_{2^{*}}^{2^{*}}&=S^{N/2}+O(\varepsilon^{N}),\\
\|u_{\varepsilon}\|_{2}^{2}&=
\begin{cases}
c \varepsilon^{2}+O(\varepsilon^{N-2}), & \text{if }N\geq5,\\
c \varepsilon^{2}|\ln \varepsilon|+O(\varepsilon^{2}), & \text{if }N=4,\\
c \varepsilon+O(\varepsilon^{2}), & \text{if } N=3,
\end{cases}
    \end{aligned}
\end{equation}
where $c$ denotes a strictly positive constant (depending on $N$). Also, we have
\begin{equation}\label{est:Up}
    \begin{aligned}\displaystyle
\|U_{\varepsilon}\|_{p}^{p}=
\begin{cases}
 c_{1}\varepsilon^{N-\frac{N-2}{2}p}+o(\varepsilon^{N-\frac{N-2}{2}p}), & \text{if }\frac{N}{N-2}<p<2^{*},\\
 c_{1}\varepsilon^{\frac{N}{2}}|\ln \varepsilon|+O(\varepsilon^{\frac{N}{2}}), & \text{if }p=\frac{N}{N-2},\\
c_{1}\varepsilon^{\frac{N-2}{2}p}+o(\varepsilon^{\frac{N-2}{2}p}), & \text{if } 1\leq p<\frac{N}{N-2}.
\end{cases}
    \end{aligned}
\end{equation}
where $c_{1}$ denotes a strictly positive constant (depending on $N,p$).

%

Recall Lemma \ref{lem:local_min_attained}, we denote $u_{\mu}$ is a ground state found in Section \ref{sec:GS} for $m_{\mu}$ with the Lagrange multiplier $\lambda_{\mu}>0$. For $N\geq 6$, we assume the assumption \eqref{assump:Nge6} holds. Hence, we have the following estimate for mountain pass level $c_{\mu}$.

\begin{lemma}\label{lem:es_case3}
Under the assumptions of Theorem \ref{thm:MP}. We have
\begin{equation}\label{eq:est_MP3}
c_{\mu}< m_{\mu}+ \frac{1}{N}S^{N/2} \mu^{1-\frac{N}{2}}.
\end{equation}
\end{lemma}

\begin{proof}
We recall the localized Aubin–Talenti function $u_{\eps}$, as defined in \eqref{eq:def_u_eps}, together with the estimates established in \eqref{est:struwe} and \eqref{est:Up}. Then, we define $W_{t,\eps}=u_{\mu}+tu_{\eps}$, where $t\geq0$. We define $s=\|W_{t,\varepsilon}\|_{2}>0$ and
\[
\tilde{W}_{t,\varepsilon}= s^{\frac{N-2}{2}}W_{t,\eps} (sx).
\]
Thus,
\[
\tilde{W}_{t,\varepsilon}=\tilde{u}_{\mu}+t\tilde{u}_{\varepsilon},\qquad\text{ with }\|\tilde{W}_{t,\varepsilon}\|_{2}=1.
\]
Moreover, we have
\begin{equation}\label{eq:def_s}
  s^{2}=\|u_{\mu}+tu_{\varepsilon}\|_{2}^{2}= 1+t^{2}\|u_{\varepsilon}\|_{2}^{2}+2t \int_{\R^{3}} u_{\mu}u_{\varepsilon}\,dx,
\end{equation}
and
\begin{equation}\label{est_s}
\begin{aligned}
  s-1=\frac{s^{2}-1}{s+1}=&   \frac{1}{s+1}\left\{ t^{2} \|u_{\varepsilon}\|_{2}^{2}+2t \int_{\R^{3}} u_{\mu}u_{\varepsilon}\,dx \right\}.
\end{aligned}
\end{equation}
Subsequently, we consider the relationship between $c_{\mu}$ and $E_{\mu}(\tilde{W}_{t,\varepsilon})$. Fix $\varepsilon$ and let $t\rightarrow 0$, it is evident that
\begin{equation*}
 \tilde{W}_{t,\varepsilon}\to u_\mu,\qquad\text{and }  E_{\mu}(u_{\mu})<0, \quad\|\nabla u_\mu. 
 \|_{2}\leq \bar{R}
\end{equation*}
Similarly, when $t\rightarrow +\infty$, we have
\begin{equation*}
  E_{\mu}(\tilde{W}_{t,\varepsilon})\rightarrow -\infty \ \ \mbox{with } \ \|\nabla\tilde{W}_{t,\varepsilon}\|_{2}\rightarrow +\infty.
\end{equation*}
Thus, we can take $t_{1}(\varepsilon)\ll 1$ and $t_{2}(\varepsilon)\gg 1$ such that
\begin{equation*}\label{MP_COMP}
  c_{\mu}\leq \max_{t_{1}(\varepsilon)\leq t\leq t_{2}(\varepsilon)}E_{\mu}(\tilde{W}_{t,\varepsilon}).
\end{equation*}

\underline{Case 1. $N=3,4,5$.} Since $3\le N \le 5$, we have $(1+t)^{2^{*}} \geq 1 +t ^{2^{*}}+2^{*} t +2^{*} t^{2^{*}-1}$ for every $t \geq 0$. Thus, by direct computation, we have
\begin{equation*}
  \begin{aligned}
  E_{\mu}(\tilde{W}_{t,\varepsilon})\leq & \frac{1}{2}\int_{\R^{N}}|\nabla \tilde{u}_{\mu}|^{2}dx+ t \int_{\R^{N}}\nabla \tilde{u}_{\mu}\cdot \nabla \tilde{u}_{\varepsilon}dx+\frac{t^{2}}{2}\int_{\R^{N}}|\nabla \tilde{u}_{\varepsilon}|^{2}dx\\
  & +\frac{1}{2}\int_{\R^{N}}V(x) \tilde{u}_{\mu}^{2}dx+t \int_{\R^{N}}V(x) \tilde{u}_{\mu}\cdot  \tilde{u}_{\varepsilon}dx+\frac{t^{2}}{2}\int_{\R^{N}} V(x) \tilde{u}_{\varepsilon}^{2}dx\\
  & -\frac{\mu}{2^{*}}\int_{\R^{N}}\tilde{u}_{\mu}^{2^{*}}dx-\frac{\mu t^{2^{*}}}{2^{*}}\int_{\R^{N}}\tilde{u}_{\varepsilon}^{2^{*}}dx-\mu t \int_{\R^{N}} \tilde{u}_{\mu}^{2^{*}-1} \tilde{u}_{\varepsilon}dx- \mu t^{2^{*}-1} \int_{\R^{N}} \tilde{u}_{\mu} \tilde{u}_{\varepsilon}^{2^{*}-1}dx.
  \end{aligned}
\end{equation*}
Then, following the same steps as in \cite[Lemma 5.4.]{verzini2025normalizedsolutionsnonlinearschrodinger}, we obtain
\begin{equation*}
  \begin{aligned}
  E_{\mu}(\tilde{W}_{t,\varepsilon})\leq &m_{\mu} +\frac{1}{N}S^{\frac{N}{2}}\mu^{1-\frac{N}{2}} +O(\eps^{N-2})+I+II+III,
  \end{aligned}
\end{equation*}
where
\begin{equation*}
\begin{split}
I:=&E_{\mu}(\tilde{u}_{\mu})-E_{\mu}(u_{\mu})+t \lambda_{\mu} s^{2} \int_{\R^{N}}\tilde{u}_{\mu}\tilde{u}_{\varepsilon}dx,\\
II:=&t \int_{\R^{N}}\left[V(x)-s^{2}V(sx)\right] \tilde{u}_{\mu}\cdot  \tilde{u}_{\varepsilon}dx,\\
III:=&\frac{t^{2}}{2}\int_{\R^{N}} V \tilde{u}_{\varepsilon}^{2}dx-\mu t^{2^{*}} \int_{\R^{N}} \tilde{u}_{\mu} \tilde{u}_{\varepsilon}^{2^{*}-1}dx.
\end{split}
\end{equation*}
To begin, we estimate $I$. We have
\begin{equation*}
\begin{aligned}
E_{\mu}(\tilde{u}_{\mu})-E_{\mu}(u_{\mu}) &= \frac{1}{2} \int_{\R^N} \left(s^{-2} V\left(\frac{x}{s}\right)-V(x)  \right) u_{\mu}^2 dx \\
&=\frac{1}{2}\int_{\R^N} u_{\mu}^2(x) \int^{s}_{1} \left[-2z^{-3} V\left(\frac{x}{z}\right) - z^{-3} \nabla V \left(\frac{x}{z}\right) \cdot \left(\frac{x}{z}\right) \right]dzdx\\
&=\int^{s}_{1} z^{N-3} \left(\int_{\R^N} u_{\mu}^2(x)\left[- V\left(\frac{x}{z}\right) - \frac{1}{2} \nabla V \left(\frac{x}{z}\right) \cdot \left(\frac{x}{z}\right) \right] d\left(\frac{x}{z}\right) \right) dz\\
&=-\int^{s}_{1} \int_{\R^N} z^{N-3} u_{\mu}^2(zx) \left[ V(x)+ \frac{1}{2} \nabla V(x) \cdot x \right]dx dz. 
\end{aligned}
\end{equation*}
By Lemma \ref{lem:Pohozaev} and since $u_{\mu}$ is a ground state, we know $u_{\mu}$ satisfies \eqref{eq:Pohozaev}. Consequently,
\[
\lambda_{\mu} \int_{\R^N} u_{\mu}^2 dx= \frac{1}{2} \int_{\R^N} \nabla V(x) \cdot x u_{\mu}^2 \, dx +\int_{\R^N} V(x) u_{\mu}^{2} dx.
\]
Then, we have
\[
E_{\mu}(\tilde{u}_{\mu})-E_{\mu}(u_{\mu}) =\int^{s}_{1} \left(\int_{\R^3} \left(u_{\mu}^2(x)-z^{N-3} u_{\mu}^2(zx) \right)\left[ V(x)+ \frac{1}{2}\nabla V(x) \cdot x \right]dx -\lambda_{\mu}\right) dz.
\]
Recall that the definition of $s$ \eqref{eq:def_s}, we have $s \to 1^+$ as $\eps \to 0^+$. Since $z \in [1,s]$, we have that, for $\eps$ small, for every $\delta>0$, there exists $\sigma>0$ independent of $z$, such that
\[
\left| \int_{\R^3\backslash B_{\sigma}} \left(u_{\mu}^2(x)-z^{N-3} u_{\mu}^2(zx) \right)\left[  V(x)+\frac{1}{2} \nabla V(x) \cdot x \right]dx \right| <\delta.
\]
On the other hand, by \eqref{potential:1} and \eqref{potential:special}, there exists $C_1>0$ such that
\begin{equation*}
\begin{aligned}
&\left| \int_{ B_{\sigma}} \left(u_{\mu}^2(x)-z^{N-3} u_{\mu}^2(zx) \right)\left[  V(x)+ \frac{1}{2} \nabla V(x) \cdot x \right]dx \right| \\
& \leq C_1  \|V\|_{L^{\infty}(B_{\sigma})} \int_{ B_{\sigma}} \left|u_{\mu}^2(x)-z^{N-3} u_{\mu}^2(zx) \right| dx \to 0^{+},
\end{aligned}
\end{equation*}
as $z\to 1^{+}$. This yields that, for $\varepsilon>0$ sufficiently small, we have
\[
\left| \int_{ B_{\sigma}} \left(u_{\mu}^2(x)-z^{N-3} u_{\mu}^2(zx) \right)\left[  V(x)+ \frac{1}{2}\nabla V(x) \cdot x \right]dx \right| \leq \delta.
\]
Combining the above estimates, and since $\delta>0$ is arbitrary, we deduce that
\[
E_{\mu}(\tilde{u}_{\mu})-E_{\mu}(u_{\mu}) \leq \int^{s}_{1}(3\delta-\lambda_{\mu}) dz = -\lambda_{\mu}(s-1) +o(|s-1|).
\]
Therefore, by \eqref{est_s}, we have
\begin{equation*}
\begin{aligned}
I&=E_{\mu}(\tilde{u}_{\mu})-E_{\mu}(u_{\mu})+t \lambda_{\mu} s^{2} \int_{\R^{N}}\tilde{u}_{\mu}\tilde{u}_{\varepsilon}dx\\
& \le -\lambda_{\mu}\frac{1}{s+1}\left\{ t^{2} \|u_{\varepsilon}\|_{2}^{2}+2t \int_{\R^{3}} u_{\mu}u_{\varepsilon}\,dx \right\} +t \lambda_{\mu} \int_{\R^{N}}u_{\mu}u_{\varepsilon}dx+o(|s-1|)\\
&=-\lambda_{\mu}\frac{t^{2} \|u_{\varepsilon}\|_{2}^{2}}{s+1}+\frac{s-1}{s+1}\lambda_{\mu}t  \int_{\R^{N}}u_{\mu}u_{\varepsilon}dx+o(|s-1|)\\
& \le \frac{s-1}{s+1}\lambda_{\mu}t  \int_{\R^{N}}u_{\mu}u_{\varepsilon}dx+o(|s-1|) = o(\eps^{\frac{N-2}{2}}).
\end{aligned}
\end{equation*}
Then, we turn to the estimate of $II$. By \eqref{potential:1}, we obtain
\begin{equation*}
\begin{aligned}
II=\int_{\R^N}\left[V(x)-s^2V(sx)\right] \tilde{u}_{\mu} \tilde{u}_{\eps} dx&=\int_{\R^N}\left|V(x)-V(sx)+V(sx)-s^2V(sx)\right| \tilde{u}_{\mu} \tilde{u}_{\eps} dx \\
&\leq \int_{B_R} L\left|(1-s)x\right|  \tilde{u}_{\mu} \tilde{U}_{\eps} dx + |1-s^2|\int_{B_R} \left|V(sx)\right|  \tilde{u}_{\mu} \tilde{U}_{\eps} dx\\
&\leq(s-1)LR\int_{B_R}\tilde{u}_{\mu} \tilde{U}_{\eps} dx+(s^2-1)\int_{B_R}\left|V(sx)\right|  \tilde{u}_{\mu} \tilde{U }_{\eps} dx\\
&=o(\eps^{\frac{N-2}{2}}).
\end{aligned}
\end{equation*}
Finally, we deal with $III$. From \eqref{est:struwe}, we have
\[
\int_{\R^{N}}V(x) \tilde{u}_{\eps}^{2}dx =  s^{-2} \int_{B_R}V\left(\frac{x}{s}\right)U_{\eps}^{2}dx \leq s^{-2} \|V\|_{L^{\infty}(B_{R})} \int_{B_R} U_{\eps}^{2}dx=o(\eps^{\frac{N-2}{2}}).
\]
By \eqref{est:Up} and a direct computation, we obtain
\begin{equation*}
- \mu t^{2^{*}-1} \int_{\R^{N}} \tilde{u}_{\mu} \tilde{u}_{\varepsilon}^{2^{*}-1}dx \leq -C_{1}\eps^{\frac{N-2}{2}}+o(\eps^{\frac{N-2}{2}}),
\end{equation*}
where $C_{1}$ is associated with the infimum of $u_{\mu}$ in the ball $B_{R}$. Therefore,
\[
III\leq -C_{1}\eps^{\frac{N-2}{2}}+o(\eps^{\frac{N-2}{2}}).
\]
In conclusion, combining the previous estimates yields \eqref{eq:est_MP3}.

\underline{Case 2. $N\geq 6$.} Following the procedure of Case 1, we proceed to estimate $I$, $II$, and $III$. Recall that
\[
E_{\mu}(\tilde{W}_{t,\varepsilon})\leq m_{\mu} +\frac{1}{N}S^{\frac{N}{2}}\mu^{1-\frac{N}{2}} +O(\eps^{N-2})+I+II+III,
\]
where $I,II$ are the same as Step 1. However, since $N \geq 6$, we have $(1+t)^{2^{*}} \geq 1 +2^{*} t +t ^{2^{*}}$ for every $t \geq 0$, which leads to
\[
III=\frac{t^{2}}{2}\int_{\R^{N}} V \tilde{u}_{\varepsilon}^{2}dx.
\]
First, we have
\[
I \leq -\lambda_{\mu}\frac{t^{2} \|u_{\varepsilon}\|_{2}^{2}}{s+1}+\frac{s-1}{s+1}\lambda_{\mu}t  \int_{\R^{N}}u_{\mu}u_{\varepsilon}dx+o(|s-1|)=-\lambda_{\mu}\frac{t^{2} \|u_{\varepsilon}\|_{2}^{2}}{s+1}+o(\eps^2).
\]
Similarly, for $II$, we have
\[
II=o(\eps^{\frac{N-2}{2}})=o(\eps^{2}).
\]
It remains to consider $III$, for which we obtain
\[
III=\frac{t^2}{2}\int_{\R^{N}}V(x) \tilde{u}_{\eps}^{2}dx \leq \frac{s^{-2} t^2}{2}  \|V\|_{L^{\infty}(B_{R}(0))} \|u_{\varepsilon}\|_{2}^{2}.
\]
From \eqref{eq:ass_MPsec} and the estimates of $I$, $II$, and $III$, we deduce that, for $R>0$ sufficiently small,
\begin{equation*}
\begin{aligned}
E_{\mu}(\tilde{W}_{t,\varepsilon})&\leq m_{\mu} +\frac{1}{N}S^{\frac{N}{2}}\mu^{1-\frac{N}{2}} +\left(\frac{s^{-2} }{2}  \|V\|_{L^{\infty}(B_{R}(0))}-\lambda_{\mu}\frac{1}{s+1}\right)t^2\|u_{\varepsilon}\|_{2}^{2} +o(\eps^2)\\
& \leq  m_{\mu} +\frac{1}{N}S^{\frac{N}{2}}\mu^{1-\frac{N}{2}} -O(\eps^2),
\end{aligned}
\end{equation*}
which yields \eqref{eq:est_MP3}.
\end{proof}

Now, we recall a result from Ghoussoub’s book \cite[Thm. 4.1]{MR1251958}, which plays a crucial role in the proof of Lemma \ref{boundedPS}.
\begin{lemma}[{\cite[Thm. 4.1]{MR1251958}}]\label{closesequence}
Let $X$ be a Hilbert manifold and let $E_{\mu}\in C^{1}(X,\mathbb{R})$ be a given functional. Let $K\subset X$ be compact and consider a subset
\[\Gamma\subset \{\gamma\subset X:\gamma \ \text{is compact }, \ K\subset \gamma\}\]
which is invariant with respect to deformations leaving $K$ fixed. Assume that
\[\max_{u\in K}E_{\mu}(u)< c_{\mu}:=\inf\limits_{\gamma\in \Gamma}\max\limits_{u\in \gamma}E_{\mu}(u).\]
Let $(\gamma_{n})_n\subset \Gamma$ be a sequence such that
\[
\max_{u\in \gamma_{n}}E_{\mu}(u)\to c_{\mu}
\qquad\text{ as }n\to+\infty.
\]
Then there exists a sequence $(v_{n})_n\subset X$ such that, as $n\to+\infty$,
\begin{enumerate}
\item $E_{\mu}(v_{n})\to c_{\mu}$,
\item $\|\nabla_X E_{\mu}(v_{n})\|\to0$,
\item $dist(v_{n},\gamma_{n})\to 0$.
\end{enumerate}
\end{lemma}

In the following lemma, we construct a bounded Palais–Smale sequence of “almost positive” functions. The proof follows the approach of \cite[Proposition 3.4]{MR4443784}.

\begin{lemma}\label{boundedPS}
There exists a Palais-Smale sequence $(v_{n})_{n}$ for $E_{\mu}$ constrained on $X$ at the level $c_{\mu}$, namely
\begin{equation}\label{PSlimit}
  E_{\mu}(v_{n})\rightarrow c_{\mu}, \ \ \ \ \nabla_{\mathcal{H}}E_{\mu}(v_{n})\rightarrow 0, \ \ \ \mbox{as } \ n\rightarrow +\infty,
\end{equation}
such that
\begin{equation}\label{Pohozaev}
  \|\nabla v_{n}\|_{2}^{2}-\mu \|v_{n}\|_{2^{*}}^{2^{*}}-\frac{1}{2}\int_{\mathbb{R}^{N}} \nabla V(x) \cdot x v_{n}^{2}dx\rightarrow 0, \ \ \ \mbox{as } \ n\rightarrow \infty,
\end{equation}
\begin{equation}\label{non_negative}
  \lim_{n\rightarrow +\infty}\|(v_{n})^{-}\|_{2}=0.
\end{equation}
Moreover, under the same conditions in Theorem \ref{thm:GS}, the sequence $(v_{n})_{n}$ is bounded and the associated Lagrange multipliers $\lambda_{n}$ are bounded too.
\end{lemma}

\begin{proof}
Our first step is to construct a Palais–Smale sequence that is almost everywhere nonnegative. In order to apply Lemma \ref{closesequence} to $\tilde{E}_{\mu}$ (recall \eqref{def:E_tilde}), we first verify the assumptions in Lemma \ref{closesequence}. From \eqref{level:mp}, let $\xi_{n} \in \Gamma$ be such that
\[
\max_{t \in [0,1]} E_{\mu}(\xi_{n}) \to c_{\mu}, \qquad \text{as } n \to + \infty.
\]
Since $E_{\mu}$ and $\mathcal{H}$ are even, we can take $\xi_{n} \geq 0$ in $\R^{N}$ for every $t \in [0,1]$ and $n \in \mathbb{N}$. Recall \eqref{def:E_tilde}, \eqref{def:path_tilde}, \eqref{def:MP_tilde} and Lemma \ref{ener_compare}, we have
\[
K=\{(w_{0},0), (w_{h_{1}},0)\}, \ \ \Gamma= \tilde{\Gamma}, \ \ X:=\mathcal{H} \times \R, \ \ \gamma_{n}=\{(\xi_{n}(t),0): t \in[0,1]\},
\]
and
\[
\max_{t \in [0,1]} E_{\mu}(\gamma_{n}) \to \tilde{c}_{\mu}, \qquad \text{as } n \to + \infty.
\]
Therefore, by Lemma \eqref{closesequence}, we obtain that there exists a sequence 
$(u_{n},h_{n})_{n} \in X\times \R$ such that
\begin{equation*}
  \tilde{E}_{\mu}(u_{n},h_{n})\rightarrow \tilde{c}_{\mu}, \ \ \|\nabla_{\mathcal{M}}\tilde{E}_{\mu}(u_{n},h_{n})\|\rightarrow 0 \ \ \|(u_{n},h_{n})-(\xi_{n},0)\|\rightarrow 0, \ \mbox{for some } \ t_{n}\in[0,1].
\end{equation*}
Then, we define
\begin{equation}\label{eq:redefine_v}
v_{n}(x)=e^{\frac{Nh}{2}}u_{n}(e^{h}x).
\end{equation}
By $\tilde{c}_{\mu}=c_{\mu}$ and $\xi_{n} \geq 0$, we obtain that $(v_n)_n$ satisfies \eqref{PSlimit} and \eqref{non_negative}. Then, differentiating $\tilde{E}_{\mu}$ with respect to $h$, we obtain \eqref{Pohozaev}.

Our second step is to show that the Palais-Smale sequence and the corresponding Lagrange multipliers $\lambda_{n}$  are bounded. For $(v_{n})_{n}$ defined in \eqref{eq:redefine_v}, we set
\[
a_{n}=\|\nabla v_{n}\|_{2}^{2}, \ \ b_{n}= \int_{\R^{N}} V(x) v_{n}^{2}\, dx, \ \ d_{n}= \int_{\R^{N}} \nabla V(x) \cdot x v_{n}^{2}\, dx, \ \ e_{n}=\|v_{n}\|_{2^{*}}^{2^{*}}.
\]
From \eqref{PSlimit} and \eqref{Pohozaev}, we deduce that
\begin{equation}\label{eq:bdd_enerlevel}
a_{n}+b_{n}-\frac{2\mu}{2^{*}}e_{n}= 2c_{\mu}+o(1),
\end{equation}
\begin{equation}\label{eq:bdd_mainequation}
a_{n}+b_{n}=\lambda_{n}+\mu e_{n}+o(1)((a_{n}+1)^{1/2}),
\end{equation}
\begin{equation}\label{eq:bdd_pohozaev}
a_{n}-\frac12 d_{n}-\mu e_{n}=o(1).
\end{equation}
Combining \eqref{eq:bdd_enerlevel} and \eqref{eq:bdd_pohozaev}, we obtain
\[
\frac{2\mu}{N} e_{n}= -b_{n}-\frac{1}{2} d_{n} +2c_{\mu}+o(1)< 2c_{\mu}+o(1),
\]
which yields $\|v_{n}\|_{2^{*}}^{2^{*}}=e_{n}$ is bounded. Then, by \eqref{eq:bdd_enerlevel}, we have
\[
a_{n}+b_{n}=\frac{2\mu}{2^{*}}e_{n}+2c_{\mu}+o(1)<\left( \frac{N-2}{2}\right)2c_{\mu} +2c_{\mu}+o(1)= N c_{\mu}+o(1),
\]
which implies that $\|v_{n}\|_{\mathcal{H}}^{2}=a_{n}+b_{n}+1$ is bounded. Finally, by \eqref{eq:bdd_mainequation} and since $a_{n}$, $b_{n}$ and $e_{n}$ are bounded, we obtain $\lambda_{n}$ is bounded.
\end{proof}

From this lemma together with Lemmas \ref{lem:split} and \ref{lem:es_case3}, it follows naturally that $c_\mu$ is indeed a mountain-pass level.

\begin{lemma}\label{lem:mon_minmax_attained}
Assume that the hypotheses of Theorem~\ref{thm:MP} hold. Then, for every $0 < \mu < \mu^{*}$, the mountain-pass level $c_\mu$ is achieved, that is, there exists $v_\mu \in \mathcal{M}$ such that
\[
E_\mu(v_\mu) = c_\mu.
\]
In particular, $v_\mu$ is a mountain-pass solution of $E_\mu$.
\end{lemma}
\begin{proof}
From Lemma \ref{boundedPS}, we can construct a positive bounded Palais-Smale sequence $(v_{n})_{n}$ and the associated Lagrange multiplier $\lambda_{n}$ is also bounded. Since $(v_n)_{n}$ is bounded in $\mathcal{H}$, there exists $v^0 \in \mathcal{H}$ such that, up to a subsequence,
\[
v_n \rightharpoonup v^0 \text{ in } \mathcal{H}, \quad v_n \to v^0 \text{ strongly in } L^2(\R^N), \quad \text{and } v_n \to v^0 \text{ a.e. in } \R^{N}. 
\]

To conclude the proof, it remains to show that $v_n \to v^0$ strongly in $\mathcal{H}$, which implies that $c_\mu$ is attained. We argue by contradiction, assuming that $v_n \rightharpoonup v^0$ in $\mathcal{H}$ but that the convergence is not strong. By \eqref{non_negative}, we are in a position to apply Lemma \ref{lem:split}.

Recall Lemma \ref{lem:es_case3}, we obtain the estimates for the mountain-pass energy level. More precisely, under the assumptions of Theorem \ref{thm:MP}, the following estimate holds:
\begin{equation*}\label{eq:again_mpfromabove}
c_{\mu}< m_{\mu}+ \frac{1}{N}S^{N/2}\mu^{1-\frac{N}{2}},
\end{equation*}
where $m_{\mu}$ is a ground state. By Lemma \ref{lem:split}, we have case \ref{eq:split_case2} holds, which means
\[
 c_\mu \geq E_\mu (v^{0})+\frac{1}{N}S^{N/2} \mu^{1-\frac{N}{2}}.
\]
However, since $m_{\mu}$ is a ground state, it follows naturally that
\[
E_\mu (v^{0}) \geq m_{\mu},
\]
which is a contradiction. Hence, for $N \geq 3$, $v_n \to v^0$ strongly in $\mathcal{H}$, so that $c_{\mu}$ is achieved.
\end{proof}

\textbf{Acknowledgments.} Work partially supported by: PRIN-20227HX33Z ``Pattern formation in nonlinear 
phenomena'' - funded  by the European Union-Next Generation EU, Miss. 4-Comp. 1-CUP D53D23005690006; 
the MUR grant Dipartimento di Eccellenza 2023-2027. 

We thank Prof. Gianmaria Verzini for valuable discussions and guidance.\bigskip

\textbf{Data Availability.} Data sharing not applicable to this article as no datasets were generated or analyzed during the current study.

\bigskip

\textbf{Disclosure statement.} The authors report there are no competing interests to declare.

\bibliography{normalized}{}
\bibliographystyle{abbrv}
\medskip
\small

\begin{flushright}
{\tt junwei.yu@polimi.it}\\
Dipartimento di Matematica, Politecnico di Milano\\
piazza Leonardo da Vinci 32, 20133 Milano, Italy.
\end{flushright}

\end{document}